\documentclass[12pt]{article}
\usepackage{amsthm}
\usepackage{amsmath,amsfonts,amssymb,rotating,colordvi}
\usepackage{color}
\usepackage[unicode=true,
 bookmarks=true,bookmarksnumbered=true,bookmarksopen=true,bookmarksopenlevel=2,
 breaklinks=false,pdfborder={0 0 1},backref=false,colorlinks=true]
 {hyperref}

\usepackage{tikz}

\usetikzlibrary{calc}
\usetikzlibrary{shapes.geometric}

\tikzset{
  c/.style={every coordinate/.try}
}

\usetikzlibrary{arrows,shapes,positioning}
\usetikzlibrary{decorations.markings}
\tikzstyle arrowstyle=[scale=1]
\tikzstyle directed=[postaction={decorate,decoration={markings,mark=at position 0.6 with {\arrow[arrowstyle]{stealth};}}}]
\tikzstyle reverse directed=[postaction={decorate,decoration={markings,mark=at position 0.4 with {\arrowreversed[arrowstyle]{stealth};}}}]
\tikzstyle dot=[style={circle,inner sep=1pt,fill}]

\def\qed{\nopagebreak\hfill{\rule{4pt}{7pt}}}

 \newtheorem{thm}{Theorem}[section]

\newtheorem{lem}[thm]{Lemma}
\newtheorem{prop}[thm]{Proposition}

\newtheorem{example}[thm]{Example}

\newtheorem{coro}[thm]{Corollary}
\newtheorem{conj}[thm]{Conjecture}

\numberwithin{equation}{section}

\newdimen\Squaresize \Squaresize=11pt
\newdimen\Thickness \Thickness=0.7pt
\def\Square#1{\hbox{\vrule width \Thickness
   \vbox to \Squaresize{\hrule height \Thickness\vss
    \hbox to \Squaresize{\hss#1\hss}
   \vss\hrule height\Thickness}
\unskip\vrule width \Thickness} \kern-\Thickness}

\def\Vsquare#1{\vbox{\Square{$#1$}}\kern-\Thickness}

\def\moins{\raise 1pt\hbox{{$\scriptstyle -$}}}

\newcommand\PT{\mathcal{PT}}

\newcommand\urr{\operatorname{urr}}

\newcommand\wnm{\operatorname{w\overline{m}}}

\newcommand\WNM{\operatorname{W\overline{M}}}

\newcommand\topone{\operatorname{topone}}

\newcommand\rlm{\operatorname{rlm}}
\newcommand\RLM{\operatorname{Rlm}}
\newcommand\des{\operatorname{des}}
\newcommand\asc{\operatorname{asc}}

\newcommand\rlmin{\operatorname{rlmin}}
\newcommand\lrmin{\operatorname{lrmin}}
\newcommand\lrmax{\operatorname{lrmax}}
\newcommand\rlmax{\operatorname{rlmax}}
\newcommand\ides{\operatorname{ides}}
\newcommand\A{\operatorname{A}}
\newcommand\T{\operatorname{T}}
\newcommand\N{\operatorname{N}}
\newcommand\I{\operatorname{I}}
\newcommand\LRMAX{\operatorname{Lrmax}}
\newcommand\LRMIN{\operatorname{Lrmin}}
\newcommand\RLMAX{\operatorname{Rlmax}}
\newcommand\RLMIN{\operatorname{Rlmin}}
\newcommand\ZERO{\operatorname{Zero}}
\newcommand\MAX{\operatorname{Max}}
\newcommand\st{\operatorname{st}}
\newcommand\DIST{\operatorname{Dist}}
\newcommand\zero{\operatorname{zero}}
\newcommand\dist{\operatorname{dist}}
\newcommand\DES{\operatorname{Des}}
\newcommand\IDES{\operatorname{Ides}}
\newcommand\ASC{\operatorname{Asc}}
\newcommand\Des{\operatorname{Des}}
\newcommand\comple{\operatorname{c}}
\newcommand\inver{\operatorname{i}}

\begin{document}

\begin{center}
{\large\bf  Block decomposition and  statistics arising from permutation tableaux}
\end{center}

\begin{center}
Joanna N. Chen

College of Science\\
Tianjin University of Technology\\
Tianjin 300384, P.R. China

joannachen@tjut.edu.cn

\end{center}
\begin{abstract}
Permutation statistics  $\wnm$ and $\rlm$ are both arising
from permutation tableaux. $\wnm$ was introduced by Chen and Zhou,
which was proved equally distributed with the number of
unrestricted rows of a permutation tableau. While $\rlm$ is shown
by Nadeau equally distributed with the number of
$1$'s in the first row of a permutation tableau.

In this paper, we investigate the joint distribution of $\wnm$
and $\rlm$. Statistic $(\rlm,\wnm,\rlmin,\des,(\underline{321}))$  is shown equally distributed
with $(\rlm,\rlmin,\wnm,\des,(\underline{321}))$ on $S_n$. Then the generating function of $(\rlm,\wnm)$ follows. An involution is constructed to
explain the symmetric property of the generating function.
Also, we study the triple statistic $(\wnm,\rlm,\asc)$, which is
shown to be equally distributed with $(\rlmax-1,\rlmin,\asc)$
as studied by Josuat-Verg$\grave{e}$s.
The main method we adopt throughout the paper is  constructing
bijections based on a block decomposition of permutations.
\end{abstract}


\section{Introduction}

In this paper, we mainly investigate two permutation statistics $\wnm$ and $\rlm$  which are arising from permutation tableaux.

Permutation tableaux were introduced by Steingr\'{\i}msson and Williams \cite{Einar}. They are related to the enumeration of totally positive Grassmannian cells \cite{LW,Postnikov,Scott,Williams} and a statistical physics model called Partially Asymmetric Exclusion Process (PASEP) \cite{Corteel,CorteelWilliamsAAM,CorteelWilliamsInt,CorteelWilliamsDuke,
CorteelBrak,Matthieu}.
Several papers on the combinatorics of permutation tableaux  have also been published, see \cite{Burstein,Chen,ChenLiu,CorteelKim,CorteelNadeau,CorteelMatthieu,Nadeau}.

A \emph{permutation tableau} is a Ferrers diagram with possibly empty rows together with a $0,1$-filling of the cells satisfying the following conditions:

1. each column has at least one $1$,

2. there is no $0$ which has a $1$ above it in the same column and a $1$ to the left of it in the same row.

The \emph{length} of a permutation tableau is defined to be the number of rows plus the number of columns. Let $\PT(n)$ denote the set of permutations
of length $n$.

Several statistics over permutation tableaux are defined,
among which $\urr$ and $\topone$ are two interesting ones.
 A $0$ in a permutation tableau
is \emph{row-restricted} if there is a $1$ above in the same column.
A row is said to be \emph{unrestricted} if it  contains  no row-restricted
$0$. For $T\in \PT(n)$, as given by Corteel and Kim \cite{CorteelKim},
let $\urr(T)$  be the number of unrestricted rows  of $T$ and
 let $\topone(T)$  be the number of $1$'s in the first row of $T$.
Using recurrence relations, Corteel and Nadeau \cite{CorteelNadeau}
 obtained an explicit formula for the generating function of permutation tableaux of length $n$ with respect to the statistics $\urr$ and $\topone$.  Corteel and Kim \cite{CorteelKim} rewrote this formula
as follows
\begin{equation}\label{equ:corteelNadeau}
  \sum_{T\in\mathcal{PT}(n)} x^{\urr(T)-1} y^{\topone(T)} = (x+y)_{n-1},
\end{equation}
where  $(x)_n=x(x+1)\cdots(x+n-1)$ for $n\geq1$ with $(x)_0=1$.
Moreover, they gave two beautiful bijective proofs of
(\ref{equ:corteelNadeau}).

Permutation tableaux are in bijections with
permutations. Let $[n]=\{1,2, \ldots, n \}$ and  $S_n$ be the set of permutations on $[n]$. Steingr\'{\i}msson and Williams \cite{Einar} gave a Zig-Zag map $\Phi$ from $\PT(n)$ to $S_n$.
Given $T\in \PT(n)$, we label $T$ as follows. First, label the steps in the south-east border with $1,2, \cdots,n$ from north-east
to south-west. Then, label a row (resp. column) with $i$ if the row contains the south (resp. west) step with label $i.$
A zigzag path on a permutation tableau is a path
entering from the left of a row or the top of a column,
going to the east or to the south changing the direction alternatively
whenever it meets a $1$ until exiting the tableau.
Let $\pi=\pi_1 \pi_2 \cdots \pi_n =\Phi(T)$, where $\pi_i=j$
if the zigzag path corresponding to label $i$ exits $T$ from a row or a column labeled by $j$.
As an example, $\Phi(\pi)=8,6,1,5,3,4,9,2,7,11,10$ for $\pi$ given in the left of Figure \ref{fig:tableau}.

\begin{figure}[!htbp]
\begin{center}
\begin{tikzpicture}[line width=0.7pt,scale=0.6]
\coordinate (O) at (0,0);
\coordinate (C) at (14,0);
\draw[thick] (O)--++(0,-5);
\draw[thick](O)++(1,0)--++(0,-5);
\draw[thick](O)++(2,0)--++(0,-4);
\draw[thick](O)++(3,0)--++(0,-4);
\draw[thick](O)++(4,0)--++(0,-3);
\draw[thick](O)++(5,0)--++(0,-3);
\draw[thick](O)++(6,0)--++(0,-2);
\draw[thick](O)--++(6,0);
\draw[thick](O)++(0,-1)--++(6,0);
\draw[thick](O)++(0,-2)--++(6,0);
\draw[thick](O)++(0,-3)--++(5,0);
\draw[thick](O)++(0,-4)--++(3,0);
\draw[thick](O)++(0,-5)--++(1,0);

\path (O)++(-0.5,-0.5) node {$1$} ++(1,0) node {$0$} ++(1,0) node {$1$}
++(1,0) node {$1$} ++(1,0) node {$0$}++(1,0) node {$0$}++(1,0) node {$1$};
\path(O)++ (-0.5,-1.5) node {$2$} ++(1,0) node {$0$} ++(1,0) node {$0$}
++(1,0) node {$0$} ++(1,0) node {$1$}++(1,0) node {$1$}++(1,0) node {$1$};
\path (O)++(-0.5,-2.5) node {$4$} ++(1,0) node {$0$} ++(1,0) node {$0$}
++(1,0) node {$0$} ++(1,0) node {$0$}++(1,0) node {$1$};
\path (O)++(-0.5,-3.5) node {$7$} ++(1,0) node {$0$} ++(1,0) node {$1$}
++(1,0) node {$1$};
\path (O)++(-0.5,-4.5) node {$10$} ++(1,0) node {$1$};
\path (O)++(0.5,0.5) node {$11$} ++(1,0) node {$9$}++(1,0) node {$8$}
++(1,0) node {$6$}++(1,0) node {$5$}++(1,0) node {$3$};

\draw[thick] (C)--++(0,-5);
\draw[thick](C)++(1,0)--++(0,-5);
\draw[thick](C)++(2,0)--++(0,-4);
\draw[thick](C)++(3,0)--++(0,-4);
\draw[thick](C)++(4,0)--++(0,-3);
\draw[thick](C)++(5,0)--++(0,-3);
\draw[thick](C)++(6,0)--++(0,-2);
\draw[thick](C)--++(6,0);
\draw[thick](C)++(0,-1)--++(6,0);
\draw[thick](C)++(0,-2)--++(6,0);
\draw[thick](C)++(0,-3)--++(5,0);
\draw[thick](C)++(0,-4)--++(3,0);
\draw[thick](C)++(0,-5)--++(1,0);
\path (C)++(-0.5,-0.5) node {$1$} ++(2,0)
node{$\uparrow$} ++(1,0)
node{$\uparrow$}++(3,0)node {$\uparrow$};
\path(C)++ (-0.5,-1.5) node {$2$}
++(3,0) node {$\leftarrow$} ++(1,0) node {$\uparrow$}++(1,0) node {$\uparrow$};
\path (C)++(-0.5,-2.5) node {$4$} ++(4,0) node {$\leftarrow$};
\path (C)++(-0.5,-3.5) node {$7$} ;
\path (C)++(-0.5,-4.5) node {$10$} ++(1,0) node {$\uparrow$};
\path (C)++(0.5,0.5) node {$11$} ++(1,0) node {$9$}++(1,0) node {$8$}
++(1,0) node {$6$}++(1,0) node {$5$}++(1,0) node {$3$};

\end{tikzpicture}
\caption{A permutation tableau (left) and its corresponding alternative representation (right).}
\label{fig:tableau}
\end{center}
\end{figure}

 Corteel and Nadeau \cite{CorteelNadeau} found another two
bijections between permutation tableaux and  permutations,
one of which we denote by $\Gamma$ is given depend on the alternative
representation of permutation tableaux.
As given in Corteel and Kim \cite{CorteelKim}, the alternative representation
of  a permutation tableau $T$ is the diagram obtained from $T$ by replacing the topmost
$1$'s by $\uparrow$'s and the rightmost restricted $0$'s by $\leftarrow$'s and removing the remaining $0$'s and $1$'s, see
Figure \ref{fig:tableau} as an example.  Given $T \in \PT(n)$,  $\pi=\pi_1 \pi_2 \cdots \pi_n=\Gamma(T)$ can be obtained as follows

1. write down the labels of the unrestricted rows of $T$ in increasing order,

2. for each column $i$ from left to right, if row $j$ contains a $\uparrow$ in column $i$ and $i_1, i_2, \cdots, i_r$$(i_1< \cdots <i_r)$ contains a $\leftarrow$ in column $i$, then add $i_1,i_2,\cdots, i_r,i$
in increasing order before $j$ in $\pi$.

As an example, $\Gamma(\pi)=9,4,6,5,2,8,3,1,7,11,10$ for $\pi$ given in the right of Figure \ref{fig:tableau}.

Statistics $\wnm$ and $\rlm$ on permutations are closely related to statistics $\urr$ and $\topone$
by $\Phi$ and $\Gamma$.
We  present the definitions of  $\wnm$ and $\rlm$ first.
Given $\pi =\pi_1 \pi_2 \cdots \pi_n$ $ \in S_n$, the index $i$ is said to be a  \emph{weak excedance} of  $\pi$
if $\pi_i \geq i$. Otherwise, it is called a \emph{non-weak excedance}.
An index $i$ is called a
\emph{mid-point} of $\pi$ if there exist $j<i$ and $k>i$ such that $\pi_j>\pi_i >\pi_k$. Otherwise, $i$ is called a \emph{non-mid-point}.
Let $\wnm(\pi)=|\WNM(\pi)|$ and
\[\WNM(\pi)= \{\pi_i| ~i\ \text {is a weak excedance and a non-mid-point of $\pi$ } \}. \]
For a word $w=w_1 w_2 \cdots w_n$ of distinct integers, $w_i$ is called a \emph{$RL$-maximum} of $w$
if $w_i >w_j$ for all $j >i$.
While $w_i$ is called a \emph{$LR$-maximum},
if $w_i > w_j$ for all $j <i$.
 The $RL$-minimum and $LR$-minimum  can be defined similarly.
Let $\RLMAX(w)$, $\LRMAX(w)$, $\RLMIN(w)$  and $\LRMIN(w)$ be the set of
$RL$-maxima, $LR$-maxima, $RL$-minima and $LR$-minima of $w$, respectively.
Set $\rlmax(w)$, $\lrmax(w)$, $\rlmin(w)$ and $\lrmin(w)$
to be  the corresponding numerical statistics.
Let $\RLM(\pi)$ be the set of the RL-maxima of the subword of $\pi$ which is
 to the left of the $1$ in $\pi$. Write $\rlm(\pi)=|\RLM(\pi)|$.
As an example, for $\pi=6, 5, 1, 10, 4, 3, 8, 9, 2, 11, 7, 12$,
we have  $ \WNM(\pi)=\{6,10,11,12 \}$, $\wnm(\pi)=4$,
$\RLM(\pi)=\{5,6\}$ and $\rlm(\pi)=2$.

Corteel and Nadeau \cite{CorteelNadeau}, Nadeau \cite{Nadeau}, Chen and Zhou \cite{Chen}
proved the first, the second and the third item in the following proposition, respectively.

\begin{prop}
For $T \in \mathcal{PT}(n)$,  let $\Gamma(T)=\pi$ and  $\Phi(T)=\sigma$, then
\begin{itemize}
  \item[1.] $\urr(T)=\rlmin(\pi)$;
  \item[2.] $\topone(T)=\rlm(\pi)$;
  \item[3.] $\urr(T)=\wnm(\sigma)$.
\end{itemize}
\end{prop}
It can be checked that  $\topone(T) \neq \rlm(\Phi(T))$. So it is interesting to investigate the distributions of $(\rlm,\wnm)$.
On the other hand, from the perspective of permutations,
we see that the definition of $\rlm$ is closely related to
$\rlmax$, while $\wnm$ is indeed the statistic $\lrmax$ (proved in Lemma \ref{wnm}). $\rlmax$ and $\lrmax$ are interesting  Stirling statistics over permutations and have been widely studied. So this is another motivation of our work.
More definitions and notations needed in this paper are listed as follows.

Given $\pi=\pi_1 \pi_2 \cdots \pi_n \in S_n$,
 its \emph{reverse} $\pi^r \in S_n$ is given
by  $\pi^r(i)=\pi(n+1-i)$. Its \emph{complement} $\pi^c$ is given by
$\pi^c(i)=n+1-\pi(i)$. Let $\pi^{-1}$  denote the
inverse of $\pi$, where $\pi^{-1}(j)=i$ if and only if $\pi(i)=j$.
For convenient, we also write $\comple(\pi)=\pi^c$ and
$\inver(\pi)=\pi^{-1}$.
Assume that  $W=\{i_1, i_2, \ldots,i_n\}$ with $i_1< i_2 < \cdots <i_n$.
Given a permutation $w$ of $W$, we define $\st(w)=(\sigma,W)$, where
$\sigma \in S_n$ and $\sigma$ is order-isomorphic to $w$.
Conversely, set $\st^{-1}(\sigma,W)=w$. As an example,
$\st(31574)=(21453,\{1,3,4,5,7\})$ and $\st^{-1}(21453,\{1,3,4,5,7\})=31574$.

A \emph{descent} (\emph{ascent}) of $\pi$ is a position $i \in [n-1]$ such that $\pi_i > \pi_{i+1}$ $(\pi_i < \pi_{i+1})$.
The \emph{descent set} and the \emph{ascent set} of $\pi$
are given by
$\Des(\pi)=\{i\colon \pi_i>\pi_{i+1}\}$ and
$\ASC(\pi)=\{i\colon \pi_i<\pi_{i+1}\}$.
Let $\des(\pi)=|\Des(\pi)|$ and $\asc(\pi)=|\ASC(\pi)|$
 be the descent number and ascent number of $\pi$, respectively. Set $\ides(\pi)=\des(\pi^{-1})$.

An occurrence of a classical \emph{pattern} $p$ in a permutation $\sigma$ is a subsequence of $\sigma$ that is order-isomorphic to $p$. For instance, $41253$ has two occurrences of the pattern $3142$ in its subsequences $4153$ and $4253$. $\sigma$ is said to \emph{avoid} $p$ if there exists no occurrence of $p$ in $\sigma$.
The \emph{vincular pattern} is a generalization of the classical pattern.
Adjacent letters that are
underlined must stay adjacent when they are placed  back to the original permutation.
As an example, $41253$ now contains only one occurrence of the vincular pattern $\underline{31}42$ in its subsequence $4153$, but not in $4253$.
See \cite{Kitaev} for more details about vincular patterns. Given a vincular pattern $\tau$ and a permutation $\pi$, we denote by $(\tau)\pi$ the number of occurrences of the pattern $\tau$ in $\pi$. We write $S_n(\tau)$ as the set of permutations of length $n$ that avoid
 $\tau$.

An \emph{inversion sequence} of length $n$ is a word $s=s_1 s_2 \cdots s_n$
with $0 \leq s_i \leq i-1$ for $1 \leq i \leq n$. Let $I_n$
be the set of all inversion sequences of length $n$.
Assume that
\begin{eqnarray*}
  \ZERO(s)&=& \{i: 1 \leq i \leq n, s_i=0\},\\[2pt]
  \MAX(s)&=& \{i: 1 \leq i \leq n, s_i=i-1\}, \\[2pt]
     \DIST(s)&=& \{2 \leq i \leq n : s_i \neq 0 \,\text{and}\, s_i \neq s_j \,\text{for all}\, j >i \},
 \end{eqnarray*}
and  $\zero(s)$, $\max(s)$ and  $\dist(s)$ is the numerical statistics, respectively.

In this paper, we find that $(\rlm,\wnm)$ and $(\rlm,\rlmin)$ are
equally distributed on $S_n(321)$,  as well as on $S_n$. Particularly, we have the
following theorems.

\begin{thm}\label{thm:sn321}
Statistic $(\rlm, \rlmin, \wnm, \des,\ides)$ are equally distributed with statistic $(\rlm, \wnm, \rlmin, \des, \ides)$  over $S_n(321)$.
\end{thm}

\begin{thm}\label{thm:sn}
Statistic $(\rlm, \rlmin, \wnm, \des,(\underline{321})))$ are equally distributed with statistic $(\rlm, \wnm, \rlmin, \des,(\underline{321})))$  on $S_n$, and hence we have
\begin{equation}\label{equ:rlmwnm}
  \sum_{\pi \in S_n} x^{\wnm(\pi)-1} y^{\rlm(\pi)} = (x+y)_{n-1}.
\end{equation}
\end{thm}

Notice that $x$ and $y$ are symmetric in (\ref{equ:rlmwnm}).
We have the following theorem, which we will reprove
by an involution over $S_n$.
 \begin{thm} \label{thm:rlmlrmax}
 Statistics $(\rlm,\wnm-1)$ and $(\wnm-1,\rlm)$ are equally distributed
 on $S_n$.
\end{thm}

Josuat-Verg$\grave{e}$s \cite{Matthieu} showed that
\[
Z_N = \sum_{\pi \in S_{N+1}} \alpha^{-\rlmax(\pi)+1} \beta^{-\rlmin(\pi)+1}y^{\asc(\pi)-1} q^{(\underline{31}2) \pi },
\]
where $Z_N$ is the partition function of a partially asymmetric exclusion process (PASEP) on a finite number of sites with open and directed boundary conditions, see Theorem 1.3.1 in  \cite{Matthieu}. Inspired by this, we obtain the
following equidistribution by constructing bijections.

\begin{thm}\label{thm:rlm-rlmax-1}
Statistics $(\rlm,\wnm,\asc)$ and $(\rlm,\rlmin,\asc)$  are equally
distributed with $(\rlmax-1,\rlmin$$,\asc)$
 on $S_n$.
\end{thm}

The paper is organized as follows. In Section \ref{sec:2}, we  present bijective proofs of Theorem \ref{thm:sn321} and \ref{thm:sn}.
In Section \ref{sec:3}, we construct an involution
which implies Theorem \ref{thm:rlmlrmax}.
In section \ref{sec:4}, we bijectively prove Theorem \ref{thm:rlm-rlmax-1} by using inversion sequences.

\section{Bijective proofs of Theorem \ref{thm:sn321} and  \ref{thm:sn}}\label{sec:2}

In this section, we first deduce that statistic $\wnm$ is indeed statistic $\lrmax$. Then, two bijections based on a block decomposition of permutations
are given which imply Theorem \ref{thm:sn321} and  \ref{thm:sn}, respectively.

\begin{lem}\label{wnm}
Given $\pi \in S_n$, $\pi_i \in \WNM(\pi)$ if and only if $\pi_i$ is a LR-maximum of $\pi$.
\end{lem}

\begin{proof}
Suppose that $\pi_i \in \WNM(\pi)$, we claim that
$\pi_i$ is a LR-maximum  of $\pi$. Assume to the contrary that there exists $j <i$ such that  $\pi_j > \pi_i$. Since $\pi_i$ is a weak excedance of $\pi$, it is easily checked that there exists  $k>i$ such that $\pi_k< \pi_i$. Hence,
 $\pi_j\pi_i \pi_k$ is a $321$-pattern of $\pi$. It follows  that
 $\pi_i$ is a mid-point of $\pi$, which contradicts with the fact that $\pi_i \in \WNM(\pi)$. The claim is verified.

Conversely, assume that $\pi_i$ is a LR-maximum  of $\pi$.  Since it is larger that all the elements
to its left, then $\pi_i \geq i$ and $\pi_i$ is a non-mid-point.
It follows that $\pi_i \in \WNM(\pi)$.
This completes the proof.
\end{proof}

%
%
%

Now, we proceed to prove Theorem \ref{thm:sn321}.
As pointed out by Burstein \cite{Burstein}, $321$-avoiding permutations
have the property given in Lemma \ref{lem:321}. Based on this, Corollary \ref{coro:rlm01} follows obviously.

\begin{lem}\label{lem:321}
$\pi$ is $321$-avoiding if and only if each element of $\pi$
is  either a LR-maximum or an RL-minimum, i.e. if and only if
$\pi$ is identity or a union of two nondecreasing subsequences.
\end{lem}

\begin{coro}\label{coro:rlm01}
Given $\pi \in S_n(321)$, if $\pi_1=1$, then $\rlm(\pi)=0$. Otherwise,
$\rlm(\pi)=1$.
\end{coro}

The following lemma can be easily checked. And then, we are prepared to
give a proof of Theorem \ref{thm:sn321}.
\begin{lem}\label{lem:1n}
 Assume that $\pi=\pi_1 \pi_2 \cdots \pi_n \in S_n(321)$ with $\pi_1 \neq 1$ and $\pi_n \neq n$. Let $\sigma=\pi^{rc}$, then $\rlm(\pi)=\rlm(\sigma)=1$
 and
 \[(\rlmin, \lrmax, \des,\ides) \pi
 =(  \lrmax, \rlmin, \des,\ides) \sigma.
 \]
\end{lem}

\begin{proof}[Proof of Theorem~\ref{thm:sn321}]
It suffices to construct a bijection $\chi$ over $S_n(321)$,
which maps  $(\rlm, \rlmin, \lrmax, \des,\ides)$
to $(\rlm, \lrmax, \rlmin, \des,\ides).$

Assume that $\pi=\pi_1 \pi_2 \cdots \pi_n$ is  $321$-avoiding.
If $\pi_1 = 1$, we set $l = \max\{j \,|\, \pi_1=1, \cdots, \pi_j=j\}$, otherwise,  $l=0$.
If $\pi_n = n$, we set $r = \min\{j\, |\, \pi_j=j, \cdots, \pi_n=n\}$,
otherwise, $r=n+1$.

Let $p=p_1 \cdots
p_{n}= \chi(\pi)$ be the permutation  $1\,2\cdots l \,h
\,r \, r+1 \cdots n,$ where
\[h=\st((n+1-\pi_{r-1}) (n+1-\pi_{r-2}) \cdots (n+1-\pi_{l+1}),\{l+1,\cdots, r-1\}).\]
It should be noted that $12\cdots l$ is assumed to be empty if $l=0$,
while $r \, r+1 \cdots n$  is assumed to be empty if $r=n+1$.
Clearly, we have $p \in S_n(321)$.  Based on  Lemma \ref{lem:1n}, it can be  easily verified that
\[(\rlmin, \lrmax, \des,\ides)(\pi)=
(\lrmax, \rlmin, \des,\ides)(p).\]
By Corollary \ref{coro:rlm01}, we have
$\rlm(p)=\rlm(\pi)=0$ if $l=0$, while $\rlm(p)=\rlm(\pi)=1$ if $l>0$.
This completes the proof.
\end{proof}

\begin{example}
Let $\pi=123468759$, then $l=4$, $r=9$ and $h=\st(5324,\{5,$ $6,7,8\})=8657$. Thus we have $p=\chi(\pi)=123486579$.
It is easy to check that $\rlmin(\pi)=\lrmax(p)=6$, $\lrmax(\pi)=\rlmin(p)=7$,  $\des(\pi)=\des(p)=2$ and
$\ides(\pi)=\ides(p)=2$.
\end{example}

In the remaining part of this section, we are dedicated to proving Theorem \ref{thm:sn}.
Given a permutation $\pi$, put a bar after each RL-minimum,  and then
put a bar before  each LR-maximum if there is no bar before it. Thus we obtain a block decomposition
of $\pi$. Write the block decomposition of $\pi$ as $B_1 B_2 \cdots B_k$, we define
\begin{eqnarray*}
\N(\pi) &=&\{B_i\,| B_i~ \text{contains  neither LR-maximum nor RL-minimum} \},\\[5pt]
\T(\pi)&=& \{B_i\,| B_i~ \text{contains  both LR-maximum and RL-minimum} \},\\[5pt]
\A(\pi)&=& \{B_i\,| B_i~ \text{contains  a LR-maximum and no RL-minimum} \},\\[5pt]
\I(\pi)&=& \{B_i\,| B_i~ \text{contains  an RL-minimum and no LR-maximum} \}.
\end{eqnarray*}
For convenience, we call a block in $N(\pi)$ a N-block. The T-block,
A-block and I-block are defined similarly. Propositions of the block
decomposition below can be easily verified.

\begin{prop}\label{prop:block}
For any $\pi \in S_n$, write $\pi=B_1 B_2 \cdots B_k$, we have
\begin{enumerate}
  \item $|T(\pi)| \geq 1$.
  \item $N(\pi) \cup T(\pi) \cup A(\pi)  \cup I(\pi) = \bigcup_{1 \leq i \leq k} \{B_i\}$.
  \item If $B_i \in \N(\pi)$, then there exist integers $j<i$ and  $h>i$ such that $B_{j} \in \T(\pi)$, $B_{h} \in \T(\pi)$,
       $\{B_{j+1}, \cdots, B_{i-1}\} \subset \I(\pi)$ and
       $\{B_{i+1}, \cdots, B_{h-1}\} \subset \A(\pi)$.
  \item Let $\T(\pi) \cup \A(\pi)=\{B_{x_1},\cdots,B_{x_h}\}$, where
  $x_1 < \cdots <x_h$, then
  \[\max(B_{x_1})< \cdots < \max(B_{x_h}).\]
  \item Let $\T(\pi) \cup \I(\pi)=\{B_{x_1},\cdots,B_{x_h}\}$, where
  $x_1 < \cdots <x_h$, then
  \[\min(B_{x_1})< \cdots < \min(B_{x_h}).\]
\end{enumerate}
\end{prop}

  In the following, we define two operations on permutations. Given $\pi=\pi_1 \pi_2 \cdots \pi_n \in S_n$, assume that $\min(\pi)=\pi_i$ and $\max(\pi)=\pi_j$, let
 \begin{eqnarray*}
   L(\pi) &=& \pi_{i+1} \cdots \pi_n \pi_1 \cdots \pi_i, \\[3pt]
   R(\pi) &=& \pi_j \cdots \pi_n \pi_1 \cdots \pi_{j-1}.
 \end{eqnarray*}

 \begin{prop}\label{prop:LR}
 For $\pi=\pi_1 \cdots \pi_n \in S_n$, we have
 \begin{enumerate}
   \item $R \circ L(\pi)=\pi$ if and only if $\max(\pi)=\pi_1$.
   \item $L \circ R(\pi)=\pi$ if and only if $\min(\pi)=\pi_n$.
   \item If $\max(\pi)=\pi_1$, then $(\underline{321})(\pi)=(\underline{321})(L(\pi))$.
   \item If $\min(\pi)=\pi_n$, then $(\underline{321})(\pi)=(\underline{321})(R(\pi))$.
 \end{enumerate}
 \end{prop}

  Now we are ready to present the map $\varphi$
  over $S_n$ such that for any $\pi\in S_n$
  \[(\rlm, \rlmin, \wnm, \des,(\underline{321})) \pi = (\rlm, \wnm, \rlmin, \des,(\underline{321})) \varphi(\pi).\]
  Let $\pi= B_1 B_2 \cdots B_k \in S_n$ and assume that
  \begin{equation}\label{equ:block}
  \begin{array}{l}
  N(\pi)=\,\, \{B_{N_1}, \cdots, B_{N_h}\}, \,\,\,T(\pi)=\,\, \{B_{T_1}, \cdots, B_{T_l}\},  \\[6pt]
   A(\pi)= \,\,\{B_{A_1}, \cdots, B_{A_p}\}, \,\,\,\,I(\pi)= \,\,\{B_{I_1}, \cdots, B_{I_q}\}.
  \end{array}
  \end{equation}
  If $N(\pi)=\emptyset$, then we may view $h=0$. It is similar for
  $A(\pi)$ and $ I(\pi)$.
  We can obtain $\sigma = \varphi(\pi)$  through the following there steps:
 \begin{description}
   \item[Step 1] Write down the blocks in $N(\pi)$ and $\T(\pi)$, which keeps the relative order in $\pi$, we obtain $\sigma'$;
   \item[Step 2] Insert $R(B_{I_1}), \cdots, R(B_{I_q})$ to $\sigma'$
                 by letting the maximal letter (i.e. the first letter)
                 of $R(B_{I_1}), \cdots, R(B_{I_q}), B_{T_1}, \cdots, B_{T_l}$ increase. Between two T-blocks,
                 $R(B_{I_c})(1 \leq c \leq q)$ is always to the right
                 of a N-block, if there is any. Then we obtain $\sigma''$;

   \item[Step 3] Insert $L(B_{A_1}), \cdots, L(B_{A_p})$ to $\sigma''$
                 by letting the minimal letter (i.e. the last letter)
                 of $L(B_{A_1}), \cdots, L(B_{A_p}), B_{T_1}, \cdots, B_{T_l}$ increase. Between two T-blocks,
                 $L(B_{A_d})(1 \leq d \leq p)$ is always to the left
                 of a N-block and $R(B_{A_c})(1 \leq c \leq q)$, if there is any. Then we obtain $\sigma$.
 \end{description}

\begin{example}
Let $\pi=10,2,6,11,1,8,13,3,5,9,4,12,7$, then the block decomposition
of $\pi$ is
\[
\big{|}10\,\,\,\,\,\,\,2\,\,\,\,\,\,\,6\,\,\,\,\,\,\,\big{|}11\,\,\,\,\,\,\,
1\big{|}
\,\,\,\,\,\,\,8\,\,\,\,\,\,\,\big{|}13
\,\,\,\,\,\,\,3\big{|}\,\,\,\,\,\,\,5
\,\,\,\,\,\,\,9\,\,\,\,\,\,\,4\big{|}\,\,\,\,\,\,\,12\,\,\,\,\,\,\,7\big{|}
\]
and $N(\pi)=\{8\}, T(\pi)= \{  11\,\, 1, 13\,\,3\},
     A(\pi)= \{10 \,\,2\,\,6\}, I(\pi)= \{5 \,\,9\,\,4, 12\,\,7\}.$
By the three steps given above, we have
\begin{eqnarray*}
  \sigma'&=&11\,\,1,8,13\,\,3 , \hskip 0.5cm
  \sigma''=9\,\,4\,\,5, 11\,\,1, 8, 12\,\,7, 13\,\,3  \\
  \sigma&=&9\,\,4\,\,5, 11\,\,1, 6\,\,10\,\,2,  8, 12\,\,7, 13\,\,3
\end{eqnarray*}

%
\end{example}

\begin{prop}\label{prop:newblock}
 Let $\sigma= \varphi(\pi)$, we have
\begin{enumerate}
  \item[(1)] $T(\sigma)=\{B_{T_1}, \cdots, B_{T_l}\}$;
  \item[(2)] $A(\sigma)=\{R(B_{I_1}), \cdots, R(B_{I_q})\}$;
  \item[(3)] $I(\sigma)=\{L(B_{A_1}), \cdots, L(B_{A_p})\}$;
  \item[(4)] $N(\sigma)=\{B_{N_1}, \cdots, B_{N_h}\}$.
\end{enumerate}
\end{prop}
\begin{proof}
Firstly, we wish to show that the first letter of $B_{T_j}$(i.e.$\max\{B_{T_j}\}$), where
$1 \leq j \leq l$,  is a LR-maximum of $\sigma$, while
the last letter of $B_{T_j}$ (i.e.$\min\{B_{T_j}\}$) is an RL-minimum of $\sigma$.
By definition of step $1$ in the description of
$\varphi$, we easily check that the first letter of $B_{T_j}$ is a LR-maximum of $\sigma'$.
Since the maximal letter  $R(B_{I_1}), \cdots, R(B_{I_q}), B_{T_1}, \cdots, B_{T_l}$ increase in step $2$,  we have the first letter of $B_{T_j}$ is a LR-maximum of $\sigma''$.
Assume that $L(B_{A_i})$ is to the left of $B_{T_j}$
in $\sigma$, then we have $\min\{B_{A_i}\}< \min\{B_{T_j}\}$.
It follows that $B_{A_i}$ is to the left of $B_{T_j}$ in $\pi$,
which means that $\max\{B_{A_i}\}< \max\{B_{T_j}\}$.
Above all, the first letter of $B_{T_j}$  is a LR-maximum of $\sigma$.

Now we proceed to show that  the last letter of
$B_{T_j}$  is an RL-minimum of $\sigma$.
Clearly, $\min\{B_{T_j}\}$ is an RL-minimum of $\sigma'$.
the maximal letter (i.e. the first letter)
of $R(B_{I_1}), \cdots, R(B_{I_q}), B_{T_1}, \cdots, B_{T_l}$ increase.

Assume that $R(B_{I_i})$ is to the right of $B_{T_j}$ in $\sigma$,
then $\max\{B_{T_j}\}< \max\{B_{I_i}\}$. It means that
$B_{I_i}$ is to the right of $B_{T_j}$ in $\pi$.
Hence, $\min\{B_{I_i}\} > \min\{B_{T_j}\}$.
Assume that $L(B_{A_i})$ is to the right of $B_{T_j}$ in $\sigma$,
then by the definition of step $3$, it is easily seen that
$\min\{B_{A_i}\} > \min\{B_{T_j}\}$. Hence, the last letter of
$B_{T_j}$  is an RL-minimum of $\sigma$, as desired.

Secondly, we need to show that
the first letter of $R(B_{I_j})$(i.e.$\max\{B_{I_j}\}$),
where $1 \leq j \leq q$,
is a LR-maximum of $\sigma$ and it contains no
RL-minimum of $\sigma$.
Clearly, the first letter of $R(B_{I_j})$
is a LR-maximum of $\sigma''$. Assume that
$L(B_{A_i})$ is to the left of $R(B_{I_j})$ in $\sigma$,
we wish to prove that $\max\{B_{A_i}\}< \max\{B_{I_j}\}$.
Let $B_{T_x}$ be the nearest T-block that
is to the right of $B_{A_i}$ in $\pi$, then we
have $\max\{B_{A_i}\}< \max\{B_{T_x}\}$ and $\min\{B_{A_i}\}> \min\{B_{T_x}\}$. By step $3$ in the description of
$\varphi$, $B_{T_x}$ is to the left of $L(B_{A_i})$, and
hence to the left of $R(B_{I_j})$ in $\sigma$.
It follows that $\max\{B_{T_x}\}< \max\{B_{I_j}\}$.
Thus, $\max\{B_{A_i}\}<\max\{B_{I_j}\}$.
Hence, the first letter of $R(B_{I_j})$
is a LR-maximum of $\sigma$.

Let $B_{T_y}$ be the nearest T-block that
is to the left of $B_{I_j}$ in $\pi$, then we
have $\max\{B_{I_j}\}< \max\{B_{T_y}\}$ and $\min\{B_{I_j}\}> \min\{B_{T_y}\}$.
By step $2$, $B_{T_y}$ is to the right of $R(B_{I_j})$ in $\sigma$.
It follows from the fact $\min\{B_{I_j}\}> \min\{B_{T_y}\}$ that
$R(B_{I_j})$ contains no RL-minimum of $\sigma$, as desired.

Thirdly, we wish to show that
the last letter of $L(B_{A_j})$$(1 \leq j \leq p)$
is an RL-minimum of $\sigma$ and it contains no
LR-maximum of $\sigma$.
By description of step $3$ in $\varphi$, the last
letter of $L(B_{A_j})$ (i.e.$\min\{B_{A_j}\}$)
is smaller than all letters of $T$-blocks and $N$-blocks
which are to the right of it.
Now we assume that $R(B_{I_i})$ is to the right of $L(B_{A_j})$ in $\sigma$, if
there is any,
we aim to show that $\min\{B_{I_i}\}>\min\{B_{A_j}\}$.
Let $B_{T_x}$ be the nearest $T$-block that is to the left of
$B_{I_i}$ in $\pi$. Thus, we have $\max\{B_{I_i}\}<\max\{B_{T_x}\}$
and $\min\{B_{I_i}\}>\min\{B_{T_x}\}$.
It follows from description of step $2$ that $B_{T_x}$ is to the right of $R(B_{I_i})$ in $\sigma$, and hence to the right of $L(B_{A_j})$. Thus, by step $3$, we see that $\min\{B_{A_j}\}<\min\{B_{T_x}\}$.
Hence, $\min\{B_{I_i}\}>\min\{B_{A_j}\}$ and the last letter of $L(B_{A_j})$$(1 \leq j \leq p)$
is an RL-minimum of $\sigma$.

Let $B_{T_y}$ be the nearest $T$-block that
is to the right of $B_{A_j}$ in $\pi$,
then $\max\{B_{A_j}\}<\max\{B_{T_y}\}$
and $\min\{B_{A_j}\}>\min\{B_{T_y}\}$.
By step $3$, $B_{T_y}$ is to the left of $L(B_{A_j})$
in $\sigma$. Then, $B_{A_j}$ contains no LR-maximum
follows from the fact that $\max\{B_{A_j}\}<\max\{B_{T_y}\}$, as desired.

Notice that $B_{N_j}$$(1 \leq j \leq h)$ contains
no RL-minimum nor LR-maximum of $\sigma$.  By all the
analysis above, we may obtain a block decomposition
of $\sigma$ and propositions $(1)-(4)$ follows. This
completes the proof.
\end{proof}

\begin{proof}[Proof of Theorem~\ref{thm:sn}]
Let $\pi=\pi_1 \pi_2 \cdots \pi_n \in S_n$ with
a block decomposition given in (\ref{equ:block}) and
$\sigma=\varphi(\pi)$.
It suffices to show that $\varphi$ is an
involution over $S_n$
such that
\begin{equation}\label{equ:equdis}
 (\rlm, \rlmin, \lrmax, \des,(\underline{321}))(\pi) =(\rlm, \lrmax, \rlmin, \des,(\underline{321}))(\sigma).
\end{equation}

Firstly, we prove that $\varphi$ is an
involution, i.e. $\varphi(\sigma)=\pi$.
Assume that $p=\varphi(\sigma)$,
by applying Proposition \ref{prop:newblock} twice, we have
 \begin{eqnarray*}
    N(p)&=& \{B_{N_1}, \cdots, B_{N_h}\},\\[2pt]
     T(p)&=& \{B_{T_1}, \cdots, B_{T_l}\}, \\[2pt]
     A(p)&=& \{R \circ L (B_{A_1}), \cdots, R \circ L (B_{A_p})\}, \\[2pt]
     I(p)&=& \{L \circ R (B_{I_1}), \cdots, L \circ R(B_{I_q})\}.
 \end{eqnarray*}
  Notice that $\max\{B_{A_c}\}$ is
  the first letter of $B_{A_c}$ for $1 \leq c \leq p$,
  while $\min\{B_{I_d}\}$ is
  the last letter of $B_{I_d}$ for $1 \leq d \leq q$.
  Then from items $1, 2$ in Proposition \ref{prop:LR} we deduce that
\begin{eqnarray*}
   N(p)&=\,\, \{B_{N_1}, \cdots, B_{N_h}\}, \,\,\,T(p)&=\,\, \{B_{T_1}, \cdots, B_{T_l}\}, \\[2pt]
     A(p)&= \,\,\{B_{A_1}, \cdots, B_{A_p}\}, \,\,\,\,\,I(p)&= \,\,\{B_{I_1}, \cdots, B_{I_q}\}.
\end{eqnarray*}
Comparing with the block decomposition of $\pi$ given in
(\ref{equ:block}),  we see that $p=\pi$.
Hence $\varphi^2(\pi)=\pi$ and $\varphi$ is
an involution over $S_n$.

Now, we proceed to prove (\ref{equ:equdis}).
Viewing (\ref{equ:block}) and Proposition
\ref{prop:newblock},  we have
$\lrmax(\pi)=\rlmin(\sigma)=p$ and
$\rlmin(\pi)=\lrmax(\sigma)=q$.
If $\pi_1=1$, then it is easy to check that
$\sigma_1=1$. Hence, we have $\rlm(\pi)=\rlm(\sigma)=0$.
Otherwise, suppose that $\RLM(\pi)=\{\pi_{r_1}, \pi_{r_2}, \cdots, \pi_{r_s}\}$, where $r_1 < r_2 <\cdots<r_s$. Then,
$\pi_{r_1}$ is the nearest LR-maximum of $\pi$ that
is to the left of $1$. It follows that $\pi_{r_1} \cdots 1$ is a
T-block of $\pi$. Hence, $\pi_{r_1} \cdots 1$ remains a
 T-block of $\sigma$ by Proposition \ref{prop:newblock}.
 Thus, $\RLM(\sigma)=\{\pi_{r_1}, \pi_{r_2}, \cdots, \pi_{r_s}\}$ and
 we obtain that $\rlm(\pi)=\rlm(\sigma)=s$.
 We claim that
 all descents of a permutation are always contained in blocks.
 Assume that  $i$ is a descent of $\pi$ with $\pi_i >\pi_{i+1}$,
 then $\pi_i$ is not an RL-minimum and $\pi_{i+1}$ is not a LR-maximum.
 Hence there is no bar neither after $\pi_i$  nor before $\pi_{i+1}$.
 The claim is verified. It follows directly that $\des(\pi)=\des(\sigma)$.
Combining with items $3,4$ in Proposition \ref{prop:LR},
we have $(\underline{321}) (\pi)=(\underline{321}) (\sigma)$.
This completes the proof.
\end{proof}

It should be mentioned that $\varphi$ does not keep the statistic
$\ides$. We check the following conjecture by computer for $n \leq 9$.
\begin{conj}
Statistic $(\rlm, \rlmin, \lrmax, \des,\ides, (\underline{321}))$ are equally distributed with Statistic $(\rlm, \lrmax, \rlmin, \des, \ides, (\underline{321}))$ over $S_n$.
\end{conj}

 \section{A bijective proof of Theorem \ref{thm:rlmlrmax}}\label{sec:3}
 In this section, we present an involution over $S_n$ to give a combinatorial interpretation of Theorem \ref{thm:rlmlrmax}.
 In view of Lemma \ref{wnm}, it is enough to prove the following theorem.

 \begin{thm} \label{inv}
There exists an involution $\phi$ on $S_n$
such that
\begin{align}\label{re-rlm-wnm1}
    &\rlm(\pi)=\lrmax(\phi(\pi))-1,\\
    &\lrmax(\pi)-1=\rlm(\phi(\pi)).\label{re-rlm-wnm2}
\end{align}
\end{thm}

In the following, we shall give such an involution.
We first consider some special cases.
Define
\begin{align*}
    S_n^1&=\{\pi=\pi_1 \cdots \pi_n | ~\pi_n=1\},\\
    S_n^n&=\{\pi=\pi_1 \cdots \pi_n | ~\pi_n=n\}.
\end{align*}

\begin{lem}\label{sn1n}
There is a bijection $\rho$ from $S_n^n$ to $S_n^1$, such that
\begin{align}
\label{rlm-wnm}
    &\rlm(\pi)=\lrmax(\rho(\pi))-1,\\
    &\lrmax(\pi)-1=\rlm(\rho(\pi)). \label{wnm-rlm}
\end{align}
\end{lem}

\begin{proof}
Given a permutation $\pi \in S_n^n$, assume that
$\pi_k=1$ and
$\pi=w1un$, where $w$ and $u$
can be empty.
Let $\pi_i=\max(w)$  and $j$ be the least element
such that $k<j<n$ and $\pi_j>\pi_i$,
if there exist. Then, assume that
$a=\pi_1 \cdots \pi_{i-1}$, $b=\pi_i \cdots \pi_{k-1}$,
$c=\pi_{k+1} \cdots \pi_{j-1}$ and
$d=\pi_{j} \cdots \pi_{n-1}$
Thus, we decompose $\pi$ into six blocks, namely,
$\pi=ab1cdn$. It should be noted that each of the blocks $a,b,c,d$ can be empty.
Define $\rho(\pi)$ to be  $\pi'=b^r c n d^r a^r 1$
Clearly, $\pi' \in S_n^1$. It follows that $\rho$ is a map from $S_n^n$ to $S_n^1$.

To prove that $\rho$ is a bijection, we give the inverse map of $\rho$.
Given a permutation $\tau \in S_n^1$, let $\tau=pnq1$. Both of $p$ and $q$ can be empty. If there exists, assume that  $\tau_l$ is the largest element of $p$. Let $\tau_s$ is the rightmost element
of $q$ that is larger than $\tau_l$, if there exists. Suppose that
$\tau_t=n$ where $t <n$. We decompose $\tau$ into six blocks by
setting $\tau=efngh1$, where $e=\tau_1 \cdots \tau_l$, $f=\tau_{l+1} \cdots \tau_{t-1}$, $g=\tau_{t+1} \cdots \tau_s$ and $h=\tau_{s+1} \cdots \tau_{n-1}$.
Define $\chi(\tau)$ to be the permutation $\tau'$ where $\tau'=h^r e^r 1 f g^rn$. It can be easily checked that
$\chi$ is the inverse map of $\rho$. Hence, $\rho$ is a bijection.

Next, we proceed to prove relations (\ref{rlm-wnm}) and (\ref{wnm-rlm}).
It is not hard to check that the following relations.
\begin{align*}
\rlm(\pi)&= \text{the number of LR-maxima of $b^r$,}\\
\lrmax(\pi)-1&=\text{the number of LR-maxima of $a \pi_i d$,}\\
\rlm(\pi')&= \text{the number of LR-maxima of $adn$,}\\
\lrmax(\pi')-1&=\text{the number of LR-maxima of $b^r$.}
\end{align*}
Notice that the number of LR-maxima of $a \pi_i d$ equals to the number of LR-maxima of $adn$. Hence relations (\ref{rlm-wnm}) and (\ref{wnm-rlm}) follows,
as desired.
\end{proof}

Based on Lemma \ref{sn1n},   we are now ready to give the involution $\phi$ on $S_n$.

\noindent
{\it Proof of Theorem \ref{inv}.}
 Firstly, we give the description of $\phi$.
For a permutation $\pi \in S_n$, there are two cases to consider.
\begin{itemize}
\item[Case 1:]  $1$ is to the left of $n$. Assume that $\pi=unv$ and
$(e,S)=\st(un)$.
Then $\phi$ is defined by letting $\phi(\pi)=st^{-1}(\rho(e),S)v$.
\item[Case 2:] $1$ is to the right of $n$. Assume that $\pi=p1q$
 and $(o,T)=\st(p1)$.
Then $\phi$ is defined by letting $\phi(\pi)=st^{-1}(\rho^{-1}(o),T)q$.
\end{itemize}

From the construction of $\phi$, it is easily seen that
$\phi$ is an involution on $S_n$. In the following, we proceed to
prove relations (\ref{re-rlm-wnm1}) and (\ref{re-rlm-wnm2}).

By Lemma \ref{sn1n},
$\lrmax(e)-1=\rlm(\rho(e))$ and $\rlm(e)=\lrmax(\rho(e))-1$.
By order-isomorphic, we deduce that $\lrmax(un)-1=\rlm(st^{-1}(\rho(e),S))$ and $\rlm(un)=\lrmax(st^{-1}(\rho(e),S))-1$.
Notice that in case 1,
there is no element $z$ in subword $v$ such that
$z\in \RLM(\pi)$ nor $z \in \LRMAX(\pi)$.
Thus,  $\lrmax(\pi)-1=\rlm(\phi(\pi))$ and $\rlm(\pi)=\lrmax(\phi(\pi))-1$
hold for case 1. The fact that
(\ref{re-rlm-wnm1}) and (\ref{re-rlm-wnm2}) hold for case 2 can be proved similarly and we omit it here. We complete the proof.\qed

We end this section by giving examples of bijections $\rho$ and $\phi$.
\begin{example}
Let $\pi=3\, 7\, 2\, 5\, 1\, 4\, 8\, 6\, 9$, then
\[a=3, \,\,\,b=7\,2\,5,  \,\,\,c=4,  \,\,\, d=8\,6.\]
Hence, $\rho(\pi)=5\, 2\, 7\, 4\, 9\, 6\, 8\, 3\, 1$.
Let $\sigma=3\,8\,2\,5\,1\,4\,9\,6\,10\,7$, then
\[\st(3\,8\,2\,5\,1\,4\,9\,6\,10)=(3 \, 7 \,2\, 5\, 1\, 4\, 8\, 6\, 9,\{1,2,3,4,5,6,8,9,10\})\]
and hence
$\phi(\sigma)=5\, 2\, 8\, 4\, 10\, 6\, 9\, 3\, 1\,7.$
\end{example}

\section{A bijective proof of Theorem \ref{thm:rlm-rlmax-1}}\label{sec:4}
In this section,  we first prove Lemma \ref{lem:max2rlmin} by
giving an involution $\gamma$ over the set of inversion sequences of length $n$.
This allows us to construct a bijection $\alpha$  on $S_n$ implying
Lemma \ref{lem:keepascrlmax}. Based on Lemma \ref{lem:keepascrlmax}, another bijection $\beta$ over $S_n$ is given, which proves Theorem \ref{thm:rlm-rlmax-1}.

\begin{lem}\label{lem:max2rlmin}
Statistics $(\dist,\zero,\max,\rlmin)$ and $(\dist,\zero,\rlmin,\max)$
are equally distributed over $I_n$. Particularly,
there is an involution $\gamma$ over $I_n$ such that for each $e \in I_n$ we have
\begin{equation}\label{equ:max2rlmin}
  (\dist,\zero,\max,\rlmin) e =  (\dist,\zero,\rlmin,\max) \gamma(e).
\end{equation}
\end{lem}

\begin{lem}\label{lem:keepascrlmax}
Statistics $(\asc,\rlmax,\lrmax,\rlmin)$ and $(\asc,\rlmax,\rlmin,\lrmax)$ are equally distributed
over $S_n$. Particularly,
there is an involution $\alpha$ over $S_n$ such that for each $\pi \in S_n$ we have
\begin{equation}\label{equ:max2rlminsn}
  (\asc,\rlmax,\lrmax,\rlmin) \pi =  (\asc,\rlmax,\rlmin,\lrmax) \alpha(\pi).
\end{equation}
\end{lem}


To prove Lemma \ref{lem:max2rlmin}, we construct
 $\gamma$ over $I_n$ by induction.
Let $\gamma(0)=0$. For $e=e_1 e_2 \cdots e_{n-1} e_n \in I_n$,
assume that $r'=\gamma(e_1 e_2 \cdots e_{n-1})$.
Then, $r=\gamma(e)$ is obtained by inserting $e_n$ to the $e_n+1$-th
position of $r'$.

\begin{example}
Let $e=00113213$, then $\gamma(e)$ can be obtained as follows
\begin{align*}
 0 \rightarrow  00 \rightarrow 010 \rightarrow 0110 \rightarrow 01130
 \rightarrow 012130 \rightarrow 0112130 \rightarrow 01132130.
\end{align*}
And $\gamma^2(e)=\gamma(01132130)$ can be obtained as follows
\begin{align*}
 0 \rightarrow  01 \rightarrow 011 \rightarrow 0113 \rightarrow 01213
 \rightarrow 011213 \rightarrow 0113213 \rightarrow 00113213.
\end{align*}
\end{example}

Clearly, $\gamma$ is well-defined and we can easily
verify the following propositions.

\begin{prop}\label{prop:gamma}
Let $e =e_1 e_2 \cdots e_n \in I_n$ and $r=r_1 r_2 \cdots r_n=\gamma(e)$. Then
\begin{itemize}
  \item[(1)] $e_n+1$ is the largest element in $\MAX(r)$.
  \item[(2)] Assume that $j$ is the largest element in $\MAX(e)$,
      then
      \[r=\gamma(e_1 \cdots e_{j-1} e_{j+1} \cdots e_{n})e_j.\]
\end{itemize}
\end{prop}

\noindent
{\it Proof of Lemma \ref{lem:max2rlmin}.}
It suffices to show that
$\gamma$ is an involution over $I_n$ and satisfies (\ref{equ:max2rlmin}).

We proceed to prove that $\gamma$ is an involution by induction. When $n=1$, $\gamma^2(0)=0$.
Suppose that $\gamma^2(t)=t$ for each $t \in I_{n-1}$ with $n \geq 2$.
We claim that $\gamma^2(e)=e$ for each $e \in I_n$.
By Proposition \ref{prop:gamma}, we
have $e_n+1$ is the largest in $\MAX(\gamma(e_1 e_2 \cdots e_n))$
and hence $e_n$ is the last element of $\gamma^2(e_1 e_2 \cdots e_n)$.
 Combining  the construction of $\gamma$ and (2) in Proposition \ref{prop:gamma},
we deduce  that
\begin{eqnarray*}
  \gamma^2(e_1 e_2 \cdots e_n) &=& \gamma^2(e_1 e_2 \cdots e_{n-1})e_n \\
   &=& e_1 e_2 \cdots e_{n-1}e_n.
\end{eqnarray*}
The claim is verified.  Hence, $\gamma$ is an involution.

Now, we shall prove relation (\ref{equ:max2rlmin}).
It is easy to check that $(\dist,\zero) e =  (\dist,\zero) \gamma(e)$.
It is left to show that
\begin{equation}\label{equ:max2rlmin2}
   (\max,\rlmin) e =  (\rlmin,\max) \gamma(e)
\end{equation}
Obviously, it holds for $n=1$.
Suppose that   (\ref{equ:max2rlmin2}) holds for $n-1$, where $n \geq 2$, we claim that it also holds for $n$. There are two cases to consider.
If $e_n=n-1$, then $\gamma(e_1 e_2 \cdots e_n)=\gamma(e_1 e_2 \cdots e_{n-1}) (n-1)$. Thus,
\begin{eqnarray}
  \max(e_1 e_2 \cdots e_n) &=& \max(e_1 e_2 \cdots e_{n-1})+1,\label{eq:1} \\
 \rlmin(\gamma(e_1 e_2 \cdots e_n)) &=& \rlmin(\gamma(e_1 e_2 \cdots e_{n-1}))+1.\label{eq:2}
\end{eqnarray}
Combining  (\ref{eq:1}) (\ref{eq:2}) and the  hypothesis that $\max(e_1 e_2 \cdots e_{n-1})=\rlmin(\gamma(e_1$ $ e_2 \cdots e_{n-1}))$, we deduce that $\max(e)=\rlmin(\gamma(e))$. Then $\rlmin(e)=\max(\gamma(e))$ follows from the fact that $\gamma$ is an involution.

If $e_n<n-1$, by Proposition \ref{prop:gamma},  $e_n+1$ is the
largest element of $\MAX(\gamma(e))$. It follows that
$e_n$ is not an $RL$-minimum of $\gamma(e)$. Thus, we have
\begin{eqnarray}
  \max(e_1 e_2 \cdots e_n) &=& \max(e_1 e_2 \cdots e_{n-1}),\label{eq:3} \\
 \rlmin(\gamma(e_1 e_2 \cdots e_n)) &=& \rlmin(\gamma(e_1 e_2 \cdots e_{n-1})).\label{eq:4}
\end{eqnarray}
Similarly,  in view of  (\ref{eq:3}) (\ref{eq:4}) and the hypothesis,
 we have $(\max,\rlmin) e =  (\rlmin,\max) \gamma(e)$ in this case. This completes the proof.\qed

To prove Lemma \ref{lem:keepascrlmax},
we need  the permutation code $b$, namely, a bijection between permutations and inversion sequences, given by  Baril and Vajnovszki \cite{Baril}. We  give a brief review of the code $b$ first.

An  interval $[m,n]$ with $m <n$ is the set $\{x \in \mathbb{N} \colon m \leq x \leq n \}$, where $\mathbb{N}=\{0,1,\cdots\}$.
A labeled interval is a pair $(I,l)$, where $I$ is an interval and
$l$ is an integer.
Given $\pi=\pi_1 \pi_2 \cdots \pi_n \in S_n$ and an integer $i$ with
$0 \leq i <n$, let the $i$-th slice of $\pi$, $U_i(\pi)$,
to be a sequence of labelled intervals constructed recursively by the following process. Set $U_0(\pi)=([0,n],0)$. For $i \geq 1$, assume that $U_{i-1}(\pi)=(I_1,l_1),(I_2,l_2),\cdots,(I_k,l_k)$ is the $(i-1)$-th
slide of $\pi$ and $v$  is the index such that
$\pi_i \in I_v$, then $U_i(\pi)$ is constructed as  follows.
\begin{itemize}
  \item If $\min(I_v)<\pi_i =\max(I_v)$, then $U_i(\pi)$ equals
  \[(I_1,l_1), \cdots , (I_{v-1},l_{v-1}), (J,l_{v+1}),(I_{v+1},l_{v+2}), \cdots , (I_{k-1},l_{k}),(I_{k},l_{k}+1),\]
  where $J=[\min(I_v),\pi_i-1]$.

 \item If $\min(I_v)<\pi_i <\max(I_v)$, then $U_i(\pi)$ equals
  \[(I_1,l_1), \cdots , (I_{v-1},l_{v-1})(H,l_{v}), (J,l_{v+1}),(I_{v+1},l_{v+2}), \cdots , (I_{k-1},l_{k}),(I_{k},l_{k}+1),\]
  where $H=[\pi_i+1, \max(I_v)]$ and $J=[\min(I_v),\pi_i-1]$.

  \item If $\min(I_v)=\pi_i <\max(I_v)$, then $U_i(\pi)$ equals
  \[(I_1,l_1), \cdots , (I_{v-1},l_{v-1})(H,l_{v}), (I_{v+1},l_{v+1}), \cdots , (I_{k-1},l_{k-1}),(I_{k},l_{k}+1),\]
  where $H=[\pi_i+1, \max(I_v)]$.

  \item If $\min(I_v)=\pi_i =\max(I_v)$, then $U_i(\pi)$ equals
  \[(I_1,l_1), \cdots , (I_{v-1},l_{v-1}), (I_{v+1},l_{v+1}), \cdots , (I_{k-1},l_{k-1}),(I_{k},l_{k}+1).\]
\end{itemize}
Let $b(\pi)=b_1 b_2 \cdots b_n \in I_n$, where $b_i=l_v$
such that  $(I_v,l_v)$ is a labelled interval in the
$(i-1)$-th slice of $\pi$ with $\pi_i \in I_v$.

\begin{example}
For $\pi=24135$ and $\sigma=14352$, we have $b(\pi)=00210$ and $b(\sigma)=00102$ with
\begin{align*}
  U_0(\pi)&=([0,5],0), \hspace{3.5cm}   U_0(\sigma) =([0,5],0), \\
  U_1(\pi)&=([3,5],0)([0,1],1),\hspace{1.95cm} U_1(\sigma)=([2,5],0)([0,0],1),\\
  U_2(\pi)&=([5,5],0)([3,3],1)([0,1],2), \hspace{0.4cm} U_2(\sigma)=([5,5],0)([2,3],1)([0,0],2),\\
   U_3(\pi)&=([5,5],0)([3,3],1)([0,0],3),\hspace{0.4cm}
   U_3(\sigma)=([5,5],0)([2,2],2)([0,0],3),\\
    U_4(\pi)&=([5,5],0)([0,0],4),  \hspace{1.95cm}U_4(\sigma)=([2,2],2)([0,0],4).
\end{align*}
\end{example}

 Baril and Vajnovszki  also proved a set-valued  equidistribution as follows.

\begin{lem}\label{lem:Baril}
For any $\pi \in S_n$,
\[(\DES,\IDES,\LRMAX,\LRMIN,\RLMAX) \,\pi=(\ASC,\DIST,\ZERO,\MAX,\RLMIN) \, b(\pi),\]
and so statistics $(\DES,\IDES,\LRMAX,\LRMIN,\RLMAX)$ on $S_n$
has the same distribution as $(\ASC,\DIST,\ZERO,\MAX,\RLMIN)$ on $I_n$.
\end{lem}

Let $\alpha=\comple \circ \inver \circ b^{-1} \circ \gamma \circ b \circ  \inver  \circ \comple $,
then it is easy to check that $\alpha$ is an involution on $S_n$.
Now, we are ready to give the proof of Lemma \ref{lem:keepascrlmax}.

\noindent
{\it Proof of Lemma \ref{lem:keepascrlmax}.}
Given $\pi \in S_n$, it is enough to show that
\begin{equation}\label{equ:keepascrlmax}
 (\asc,\rlmax,\lrmax,\rlmin) \, \pi = (\asc,\rlmax,\rlmin,\lrmax)\, \alpha(\pi)
\end{equation}

Notice that $\asc(\pi)=\des(\comple(\pi))$ and $\des(\pi)=\ides(\inver(\pi)).$
Combining with Lemma \ref{lem:max2rlmin} and Lemma \ref{lem:Baril},
we see that  $\asc(\pi)=\asc(\alpha(\pi))$.

Furthermore, the following properties are easy to check.
\begin{itemize}
  \item[1)] $\pi_i$ is an RL-maximum of $\pi$ if and only if $i$ is an
  RL-maximum of $\pi^{-1}$.
  \item[2)]$\pi_i$ is a LR-minimum of $\pi$ if and only if $i$ is a
  LR-minimum of $\pi^{-1}$.
  \item[3)]$\pi_i$ is a LR-maximum of $\pi$ if and only if $i$ is a
 RL-minimum of $\pi^{-1}$.
  \item[4)]$\pi_i$ is an RL-minimum of $\pi$ if and only if $i$ is an
  LR-maximum of $\pi^{-1}$.
\end{itemize}
It follows that
\begin{equation}\label{equ:inverse}
  (\rlmax,\lrmin,\lrmax,\rlmin) \pi= (\rlmax,\lrmin,\rlmin,\lrmax) \pi^{-1}.
\end{equation}
Also, we have
\begin{equation}\label{equ:complement}
 (\rlmax,\lrmin,\lrmax,\rlmin) \pi= (\rlmin,\lrmax,\lrmin,\rlmax) \pi^{c}.
\end{equation}

Based on equations (\ref{equ:inverse}), (\ref{equ:complement}),
Lemma \ref{lem:max2rlmin} and Lemma \ref{lem:Baril}, we deduce that
\[
(\rlmax,\lrmax,\rlmin) \, \pi = (\rlmax,\rlmin,\lrmax)\, \alpha(\pi),
\]
as desired. This completes the proof.
\qed

%
%
For a set $X$, let $n-X$ be the set obtained by $n$ minus each element in $X$.
We are now ready to  prove Theorem \ref{thm:rlm-rlmax-1}.

\noindent
{\it Proof of Theorem \ref{thm:rlm-rlmax-1}.}
In view of Theorem \ref{thm:sn},
$(\rlm,\wnm,\asc)$ is equally distributed with $(\rlm,\rlmin,\asc)$ on $S_n$. It is enough to construct a bijection $\beta$ over $S_n$ such that
\begin{equation}\label{equ:rlm-rlmax-1}
  (\rlm,\wnm,\asc) \pi= (\rlmax-1,\rlmin, \asc) \beta(\pi)
\end{equation}
for each $\pi=\pi_1 \pi_2 \cdots \pi_n \in S_n$.

Assume that
 $\pi=x 1 y$, where $x$ and $y$ can be empty. Let $st(x)=(\overline{x},X)$ and
$st(y)=(\overline{y},Y)$.
Then set $\beta(\pi)=\sigma$, where
$\sigma=y' n x'$,
$x'=st^{-1}(\alpha(\overline{x}), n+1-X)$ and $y'=st^{-1}( \overline{y}^{rc},n+1-Y).$

To show that $\beta$ is a bijection, it suffices to construct its inverse. Given $\sigma =w n v \in S_n$, where $w$ and $v$ can be empty.
Let $st(w)=(\overline{w},W)$ and
$st(v)=(\overline{v},V)$.
Then set $\delta(\sigma)=\pi$, where
$\pi=v' 1 w'$,
$v'=st^{-1}(\alpha(\overline{v}), n+1-V)$ and $w'=st^{-1}( \overline{w}^{rc},n+1-W).$  Notice that $\alpha$ is an involution,
$\delta$ is the inverse of $\beta$. Hence, $\beta$ is a bijection.

In the following, we proceed to prove (\ref{equ:rlm-rlmax-1}).
Notice that
$\rlm(\pi)=\rlmax(x)$ and $\rlmax(\sigma)=\rlmax(x')+1$. By (\ref{equ:keepascrlmax}), we have $\rlmax(x)=\rlmax(x')$.
It follows that $\rlm(\pi)=\rlmax(\sigma)-1$.

To prove $\wnm(\pi)=\rlmin(\sigma)$, it is enough to show that  $\lrmax(\pi)=\rlmin(\sigma)$ in view of Lemma \ref{wnm}.
Let $\lrmax_{>s}(u)$ be the number of LR-maxima of the word $u$
which are larger than $s$, and $\rlmin_{<s}(u)$ be the
 number of RL-minima of the word $u$
which are smaller than $s$.
We consider the following two cases.
\begin{itemize}
  \item   x is empty.  Thus $\pi=1y$. It follows that
  $\lrmax(\pi)=1+\lrmax(y)$ and   $\rlmin(\sigma)=1+\rlmin(y')$.
  Clearly,  $\lrmax(y)=\rlmin(y')$. Hence, we have $\lrmax(\pi)=\rlmin(\sigma)$.
  \item   x is not empty. By the block decomposition, we have $\lrmax(\pi)=\lrmax(x)+\lrmax_{>max(X)}(y)$ and
$\rlmin(\sigma)=\rlmin(x')+\rlmin_{<min(n+1-X)}(y')$.
Since $\lrmax(x)=\rlmin(x')$ and
$\lrmax_{>max(X)}(y)=\rlmin_{<min(n+1-X)}(y')$,
then $\lrmax(\pi)=\rlmin(\sigma)$ follows.
\end{itemize}

Finally, we notice that $\asc(\pi)=\asc(x)+1+\asc(y)$ and
$\asc(\sigma)=\asc(x')+1+\asc(y')$.
Since $\asc(x)=\asc(x')$ and $\asc(y)=\asc(y')$,
we deduce that $\asc(\pi)=\asc(\sigma)$.
This completes the proof. \qed

\begin{example}
Let $\pi=593721684$, then $n=9$, $x=59372$ and $y=684$.
$(\bar{x},X)=\st^{-1}(x)=(35241,\{2,3,5,7,9\})$ and
$(\bar{y},Y)=\st^{-1}(y)=(231,\{4,6,8\})$.
$\alpha(\bar{x})=51342$ can be obtained as follows
\begin{align*}
 35241 \xrightarrow{c} 31425 \xrightarrow{i} 24135
  \xrightarrow{b} 00210 \xrightarrow{\gamma}
  00102 \xrightarrow{b^{-1}}
  14352 \xrightarrow{i} 15324 \xrightarrow{c} 51342.
\end{align*}
Then, $x'=\st^{-1}(51342,\{1,3,5,7,8\})=81573$,
 $y'=\st^{-1}(312,\{2,4,6\})=624$ and $\sigma=\beta(\pi)=624981573$.
 It is easy to check that $\rlm(\pi)=\rlmax(\sigma)-1=3$,
 $\wnm(\pi)=\rlmin(\sigma)=2$ and $\asc(\pi)=\asc(\sigma)=4$.

\end{example}

\section*{Acknowledgement}
We wish to thank the referees for valuable suggestions. The author was supported by the National Natural Science Foundation of China (No.~11701420) and the
Natural Science Foundation Project
of Tianjin Municipal Education Committee (No.~2017KJ243,~No.~2018KJ193).

 \end{document}